\documentclass{amsart}
\usepackage{amsmath,amsthm,amsfonts,amscd,amssymb,latexsym,mathrsfs,graphicx,color,cite}

\def\numequation{\addtocounter{subsubsection}{1}\begin{equation}}
\def\nummultline{\addtocounter{subsubsection}{1}\begin{multline}}

\theoremstyle{plain}
\newtheorem{theorem}{Theorem}
\newtheorem{corollary}[theorem]{Corollary}
\newtheorem{lemma}[theorem]{Lemma}
\newtheorem{proposition}[theorem]{Proposition}

\theoremstyle{definition}
\newtheorem{remark}[theorem]{Remark}

\newtheorem{definition}[theorem]{Definition}
\newtheorem{conjecture}[theorem]{Conjecture}
\numberwithin{theorem}{subsection}

\newcommand{\C}{\mathbf{C}}

\newcommand{\SL}{\text{SL}}

\newcommand{\BR}{\mathbf{R}}
\newcommand{\BZ}{\mathbf{Z}}

\newcommand{\BC}{\mathbf{C}}

\newcommand{\BQ}{\mathbf{Q}}

\newcommand{\BA}{\mathbf{A}}

\newcommand{\ten}{\otimes}

\renewcommand{\ten}{\otimes}

\newcommand{\ip}[2]{\langle #1, #2 \rangle}

\newcommand{\MO}{\mathcal{O}}

\newcommand{\mS}{\mathcal{S}}
\newcommand{\mO}{\mathcal{O}}

\newcommand{\fp}{\mathfrak{p}}

\newcommand{\fg}{\mathfrak{g}}

\newcommand{\tr}{\operatorname{tr}}

\newcommand{\del}{\delta}

\newcommand{\Gal}{\text{Gal}}

\newcommand{\Kpn}{K(\fp^n)}

\renewcommand{\th}{^{\text{th}}}
\renewcommand{\phi}{\varphi}

\newcommand{\dom}{\backslash}

\newcommand{\Vol}{\text{Vol}}
\newcommand{\const}{c}

\newcommand{\BAF}{\BA_F}
\newcommand{\BAFf}{\BAF^f}
\newcommand{\BAE}{\BA_E}

\newcommand{\tf}{f}
\newcommand{\tfn}{\phi(n)}

\newcommand{\Trn}{\operatorname{Tr}(n)}
\newcommand{\Kv}{K_v}
\newcommand{\Gv}{G_v}
\newcommand{\Kvppn}{K_{v_\fp}(\fp^n)}
\newcommand{\Kfpn}{K_f(\fp^n)}
\renewcommand{\det}{\mathrm{det}}
\newcommand{\norm}{{\rm Nm}}
\newcommand{\GalEF}{\Gamma_{E/F}}
\newcommand{\conj}{c}
\newcommand{\Aut}{L^2_{\mathrm disc}}

\newcommand{\Rdisci}{R_{\mathrm {disc}, \psi}}
\newcommand{\Idisci}{I_{\mathrm {disc}, \psi}}

\newcommand{\Sdisci}{S_{\mathrm{disc}, \psi}}

\newcommand{\Idiscr}{I_{\mathrm{disc, reg}}}
\newcommand{\Rdiscr}{R_{\mathrm{disc, reg}}}

\author[M. Gerbelli-Gauthier]{Mathilde Gerbelli-Gauthier}
\email{mathilde@math.uchicago.edu}
\address{Department of Mathematics, University of Chicago,
	5734 S.\ University Ave., Chicago, IL 60637}

\title[Growth of cohomology]
{Growth of cohomology of arithmetic groups and the stable trace formula: the case of $U(2,1)$.}
\date{\today}

\begin{document}
	\begin{abstract}
		We present a strategy to use the stable trace formula to compute growth in the cohomology of cocompact arithmetic lattices in a reductive group $G$. We then implement it in the case where $G = U(2,1)$, using results of Rogwaski. This recovers a result of Marshall. 
	\end{abstract} 
\maketitle 

\tableofcontents

\section{Introduction}
Let $G$ denote the real points of an anisotropic reductive algebraic group
over a number field $F$,
let $K$ be a maximal compact subgroup of $G$ and $X$ denote the symmetric space associated to $G$,
let $\Gamma$ be a (cocompact) congruence subgroup of $G$, let $\fp$ be an ideal of $F$,
and, for each $n \geq 1$,
 let $\Gamma(\fp^n)$ denote the congruence subgroup of $\Gamma$
of full level $\fp^n$.
The goal of this paper
is to present a strategy to bound the rate
of growth of the cohomology groups $H^i( \Gamma(\fp^n)\backslash X, \C)$
as $n \to \infty$.   

If $G$ does not admit discrete series, precise rates of growth are unknown even in the case of middle (or near middle) degrees, for which
one can only expect bounds
(see e.g. \cite{CE}).   On the other hand, if $G$ does admit
discrete series, then a classical
result of de George--Wallach \cite{DGW} shows that
the dimension of middle cohomology grows proportionally to
the index $[\Gamma: \Gamma(\fp^n)],$ or equivalently, like the volume of $X(\fp^n) = \Gamma(\fp^n) \dom X$. 
We would like to obtain analogous result for degrees $i$ below
the middle (and hence also in degrees above the middle,
by Poincar\'e duality), for as many choices of $i$ and $G$
as possible.  

Matsushima's formula \cite{Ma67} expresses dimensions of cohomology in terms of automorphic representations, and Vogan-Zuckerman \cite{VZ} have classified the finite list of representations $\pi$ of $G$ which can contribute to cohomology. These representations are non-tempered, and Sarnak-Xue have conjectured that the failure of $\pi$ to belong to $L^2(G)$ controls the growth of $m(\pi, \Gamma(\fp^n))$, the multiplicity of $\pi$ at level $\Gamma(\fp^n)$. 
\begin{conjecture}[Sarnak-Xue \cite{SX91}] 
Let $p(\pi)$ be the infimum over $p\geq2$ such that $K$-finite matrix coefficient of $\pi$ are in $L^p(G)$. Then for \[ m(\pi, \Gamma(\fp^n)) \ll_\epsilon \Vol(X(\fp^n))^{(2/p(\pi))+\epsilon} \quad \text{for any $\epsilon>0.$} \]
\end{conjecture}

Our strategy is based on an application of the stable trace formula developed by Arthur \cite{A}. Now letting $G$ denote the algebraic group over $F$, this formula writes the trace $I^G_{\mathrm{disc}}(f)$ of a test
function $f$ on the discrete spectrum
of automorphic forms for $G$ as a sum $S^G(f) + \sum_H S^H(f^H),$
where $S^G$ is a {\it stable} distribution, where $H$ ranges
over the endoscopic groups of $G$, and where $f^H$ denotes
the so-called transfer of $f$ to $H$.     
This sum can furthermore be decomposed according to Arthur parameters
$\psi$ taking values in the dual group $\widehat{G}$ (for simplicity
in this discussion, we ignore the distinction between dual groups
and $L$-groups); for the $\psi$-part, we only obtain contributions
for those $H$ for which $\psi$ factors through~$\widehat{H}$.

The key point is that the (local at infinity) parameters of the cohomological representations have been 
described, following work of Adams and Johnson \cite{AJ} and Arancibia, Moeglin and Renard \cite{AMR18}.  From these results,
one can see that many cohomological Arthur parameters are 
necessarily unstable, i.e.\ factor through some $\widehat{H}$. From this characterization, we can deduce that in some cases, for test functions $f$ chosen to measure the dimension of cohomology, the character identities dictated by $\psi$ allow us to control $I^G_\psi(f)$ in terms of $S^H_\psi(f^H)$.

Now a result of Ferrari based on the fundamental lemma \cite{Fe}
shows that, up to a certain easily computable volume factor,
the transfer of the characteristic function of $\Gamma(\fp^n)$
for $G$ is equal to the characteristic function of the
corresponding group $\Gamma(\fp^n)$ for $H$.  This 
allows us to
relate the growth of the $\psi$-contribution to the cohomology
of $G$ to the growth of the limit multiplicity for $H$ of representations corresponding to $\psi^H$. 

This gives an inductive approach to computing cohomology growth,
which will yield upper bounds (and hopefully, with additional work,
and perhaps additional assumptions, lower bounds) for the below-the-middle
degrees of cohomology associated to $G$ in terms of the growth of the middle
degree cohomology associated to direct factors of various endoscopic subgroups
$H$ of~$G$.   Exactly what results we obtain in general remains to be seen;
in this paper, we focus on the case when $G = U(2,1)$.   In this case, we recover known
results of Simon Marshall \cite{Ma14} (which confirm Sarnak-Xue's conjecture in this particular case) using our proposed framework and approach.

\begin{theorem}
	\label{intro theorem} Let $\Gamma(\fp^n)$ be a cocompact arithmetic lattice in $U(2,1)$. Then \[ \dim H^1(\Gamma(\fp^n), \BC) \ll N(\fp)^{3n}. \]
\end{theorem}

It will not escape the reader familiar with the article \cite{Ma14} that it already contains many elements which appear in our argument, for example bounding growth using endoscopy, as well as relying on character identities in Arthur packets and reducing the problem to a count over elements of ray class groups. We nevertheless hope that the use of the stable trace formula will provided a systematic framework in which to understand growth of cohomology. 
 
\subsection{Structure}
In section \ref{tracegeneral}, we describe our general strategy, and especially our use of the trace formula, in more detail. Then, in section \ref{U(3)} onward, we focus our attention on the case where
$G$ is an inner form of $U(3)$. 
After introducing notation,
we recall some facts about cohomological automorphic representations in general and explain how the trace formula can be used to compute the cohomology of arithmetic lattices. We then give an explicit description of representations of $G$ with nontrivial first cohomology and explain how the stable trace formula allows us to compare them to representations an endoscopic group $H$. We conclude with trace computations on the group $H$ which allow us to obtain the desired growth asymptotics. 

\section*{Acknowledgements}
A large portion of the ideas presented here were developed through conversations with Matt Emerton, and he was credited as a coauthor on earlier versions of this paper. The author would like to Colette Moeglin for helpful conversations, and Laurent Clozel and Simon Marshall for reading earlier drafts and suggesting numerous edits. 
\section{Growth of cohomology via the trace formula} \label{tracegeneral}

In this section, we sketch how to bound the growth of cohomology using the stable trace formula. We refer the reader to Arthur's book \cite{A} for the notation and the key technical results, and assume that $G$ is a (classical) group for which the endoscopic classification of representations is established. Since the groups under consideration are anisotropic, they are not literally those considered in \cite{A}, and we mostly choose this reference for its centrality in the literature. When implementing this method in general, we are more likely to refer to \cite{KMSW} for inner forms of unitary groups and to \cite{T} in the case of orthogonal and symplectic groups.  

\subsection{Bounds using the stable trace formula}

 Our assumption that $G$ is anisotropic implies that the discrete spectrum \[ \Aut(G(F) \dom G(\BAF)) \] constitutes the entire automorphic spectrum of $G$. Instead of working with the entire space $\Aut$, we will restrict our attention to the subspace $\Aut(G(F) \dom G(\BAF))_{reg}$ consisting of representations $\pi$ whose infinitesimal character at infinity is regular. It follows from the work of Vogan-Zuckerman \cite{VZ} that this includes representations whose cohomology with constant coefficients is non-zero. 

We let $\Idiscr$ denote restriction of Arthur's trace formula to the subspace $\Aut(G(F) \dom G(\BAF))_{\textrm{reg}}$, and refer to \cite[\S 3.1]{A} for a description of the various terms appearing in $\Idiscr$. By work of Bergeron-Clozel \cite{BC}, the terms in $\Idiscr$ corresponding to proper Levi subgroups $M \subset G$ all vanish. Thus if we denote the trace of the right-regular representation of $G(\BAF)$ on $\Aut(G(F)\dom G(\BAF))_{reg}$ by $\Rdiscr$, we have for any test function $f$ that \[ \Idiscr(f) = \Rdiscr(f). \]  

Let $\psi$ be a global Arthur parameter for $G$, and assume that the representations in the packet $\Pi_\psi$ have regular infinitesimal character at infinity. Then, following \cite[\S 3.3]{A}, we denote by $\Idisci$ the $\psi$-isotypical component of $\Idiscr$, and similarly for $\Rdiscr$, and we have \[ \Idisci(f) = \Rdisci(f). \]

Next, denote by $\mS_\psi$ the component group of the centralizer of $\psi$ in $\hat{G}$. It is a finite product of copies of $\BZ/2\BZ$ and contains a distinguished element $s_\psi$: the image of the matrix $-I$ in the restriction of $\psi$ to $SL_2$. The elements of $\mS_\psi$ are in correspondence with the endoscopic data $H$ appearing in the stabilization of $\Idisci$. The stabilization of the $\psi$-part of the trace formula is a decomposition: \[ \Idisci(f) = \sum_{s \in \mS_\psi} \Sdisci^{H_s}(f^{H_s}), \] where each $\Sdisci^{H_s}$ is a stable distribution on $H_s$. This distribution is described by Arthur in the \emph{stable multiplicity formula} \cite[Theorem 4.1.2]{A}. A study of the factors appearing in the stable multiplicity formula shows that up to a positive constant $N_{s,\psi}$ depending both on $s$ and $\psi$, but bounded independently of either, one can express \[ \Sdisci^{H_s}(f^H) =  N_{s,\psi}\sum_{\pi \in \Pi_{\psi} }  \epsilon_\psi^G(s_\psi s) \ip{s_\psi s}{\pi} \tr(\pi(f)).  \] Here the inner product $\ip{\;}{\;}$ is the one shown by Arthur to govern the structure of the packet $\Pi_{\psi}$, the character $\epsilon_\psi^G$ is defined in terms of symplectic root numbers, and $\tr(\pi(f))$ is the trace of $f$ on the representation $\pi$. 

The following simple observation drives our approach to limit multiplicity: both $\epsilon_\psi^G$ and $\ip{ }{\pi}$ take the value $1$ on the identity element of the group $\mS_\psi$. Since all elements in $\mS_\psi$ have order $2$, the group $H_s$ for which $s_\psi s$ is the identity will be $H_{s_\psi}:=H_\psi$. Given a function $f$ chosen to have non-negative trace on the $\pi \in \Pi_{\psi}$, the stable contribution of the group $H_\psi$ is the only entirely positive (and thus the largest) one. In short, up to the constants $N_{s,\psi}$, we have the inequality: \begin{equation} \label{mother inequality} \Idisci(f) \leq |\mS_\psi| \Sdisci^{H_\psi}(f^{H_\psi}).   \end{equation} When $G = U(2,1)$, this inequality appears in Theorem \ref{ineq}. 

\subsection{Applications to growth of cohomology}

We now discuss how to apply the inequality \eqref{mother inequality} to growth of cohomology. The classification of cohomological representations by Vogan-Zuckerman reduces the list of possible $\pi_{\infty}$ (and thus of possible $\psi_\infty$) under consideration to a finite number. We will thus only consider $\psi$ such that $\psi_\infty$ is associated to cohomological representations. For those parameters, the Arthur $SL_2$, and by extension the group element $s_\psi$, have been computed by Adams-Johnson \cite{AJ}. 

We wish to choose test functions that will compute the dimension of cohomology of arithmetic groups. By Matsushima's formula \cite{Ma67} \[\dim(H^1(Y(\fp^n), \BC)) = \sum_{\pi=\pi_{\infty}\pi_f}m(\pi)\dim(H^1(\fg,K;\pi_\infty))\dim(\pi_f)^{K_f(\fp^n)}, \] we need to pick $f(\fp^n) = f_\infty f_f(\fp^n)$ such that $f_\infty$ detects the representations with cohomology, and $f_f$ counts the number of $K(\fp^n)$-fixed vectors. 

Having already restricted our attention to a finite list of possible $\pi_\infty$, namely the ones appearing in the packets $\Pi_{\psi_\infty}$, we can use independence of characters to choose functions $f_\infty$ which detect only the representations we are interested in. At the finite places, we simply let $f_f(\fp^n)$ be the characteristic function of the appropriate open compact subgroup of $G(\BAF)$. 

A key result behind our method is the ability to control the transfer $f_f^{H_\psi}(\fp^n)$. Specifically, at the (unramified) place  $v_\fp$ corresponding to the prime $\fp$, the function $f_{v_\fp}(\fp^n)$ is the characteristic function of a congruence subgroup of full level $\fp^n$. By a result of Ferrari \cite{Fe}, the transfer $f^{H_\psi}_{v_\fp}(\fp^n)$ is, up to an explicit factor, the indicator function of the corresponding congruence subgroup of full level $\fp^n$ on $H_\psi$. Together with the fundamental lemma \cite{Ngo}, this allows us to interpret $\Sdisci^{H_\psi}(f^{H_\psi}(\fp^n))$ as counting level $K^H(\fp^n)$-fixed vectors.  See Theorem \ref{Ferrari} for a discussion of these questions when $G$ is a unitary group of rank three.

As for the bounds this method can provide, there are two extremes: 
\begin{itemize}
\item[(i)] When $H_\psi = G$, the method gives nothing since it bounds the growth of a distribution on $G$ in terms of another distribution on $G$. A special case of this, in which the growth is already known, are the representations contributing to $H^0$, as well as those with cohomology in the middle degree. 
\item[(ii)] When $H_\psi$ is a proper endoscopic group of $G$ and the parameter $\psi_{H_\psi}$ is stable for $H_\psi$, then $S^{H_\psi} = I^{H_\psi}$ and we may be able to inductively assume the growth to be known. In the particular case where the representations of $H_\psi$ corresponding to the parameter $\psi^{H_\psi}$ are the product of a discrete series representation with a character, we can use results of Savin \cite{Sa89} to bound multiplicity growth in terms of multiplicities of discrete series on smaller groups.  
\end{itemize}
We expect that this last situation is common enough for us to use our method to obtain bounds on growth of certain degrees of cohomology. More generally, this will provide bounds on limit multiplicities of certain non-tempered representations. 

\section{Rank $3$ unitary groups} \label{U(3)}
In the following sections,
we implement the preceding strategy in the case of $G$ such
that $G(\BR) = U(2,1)$.   Rather than using the general form of the
stabilized trace formula for unitary groups (which is the subject of
\cite{KMSW}), we use the results
of Rogawski \cite{R}, which have the same essential information,
although described in a somewhat more concrete (and perhaps more {\em ad hoc})
form.

The book \cite{R} doesn't use the language of Arthur parameters
that we used in the preceding section, although
its results can be interpreted that way.  In what follows,
after describing the basic framework in which we work,
we recall the various results from \cite{R} that mirror the 
pieces of the general strategy outlined in the preceding section.

\subsection{Number fields}Let $F$ be a totally real number
 field of degree $d$ and let $E/F$ be a CM extension.  
 Write $\mO_F$, $\mO_E$ to denote their rings of integers.   
 Denote the places of $F$ and $E$ by $v$ and $w$ 
 respectively, and the corresponding completions by
  $F_v$ and $E_w$. If $v$ is a place of $F$, let $E_v = E \ten_F F_v$. If $G$ is any algebraic group defined over $F$, we denote $G(F_v)$ by $G_v$. 

 Let $\BAF$ and $\BAE$ be the rings of ad\`eles of
  $F$ and $E$ and let $\norm: \BAE \to \BAF$ denote
   the norm map. Let $\BAFf$ denote the finite ad\`eles of $F$. If $G$ is an algebraic group over $F$ and $H$ is a subgroup of $G(\BAF)$, we use the notation $H_f = H \cap G(\BAFf)$. 
   
   Denote the id\`eles of norm $1$ by
    $\BAE^1$, and let $I_E^1 = E^1 \dom \BAE^1$.
     Fix a character $\mu$ of $\BAE^\times /E^\times$ 
     whose restriction to $\BAF^\times$ is the character 
     associated to $E$ by class field theory. Its restriction 
     to $E_w$ will be denoted $\mu_w$, or by $\mu$ when 
     the context is clear.  Let $\GalEF = \Gal(E/F)$ and denote
      its nontrivial element by $\conj$. We will also use the 
      notation $\conj(x) = \bar{x}$.

Fix the following finite subsets of places of $F$: \begin{itemize}
	\item $S_f$ is a set of finite places of $F$, containing the places 
	which ramify in $E$ as well as the places below those where $\mu$ is ramified.
	\item $S_\infty$ is the set of all infinite places of $F$. 
	\item Fix an infinite place $v_0$ and let $S_0 = S_\infty \setminus \{v_0\}$. 
	\item $S = S_f \cup S_\infty$.
\end{itemize}

\subsection{Unitary groups}
Temporarily, let $E/F$ be any quadratic extension, 
local or global. Let $\Phi_N$ be the matrix whose 
$ij\th$ entry is $(-1)^{i+1}\del_{i,N+1-j}$. The matrix 
$\Phi_N$ defines a hermitian form if $N$ is odd, and if 
$N$ is even and $\xi$ is an element of $E$ such that 
$\tr_{E/F}(\xi)=0$, the matrix $\xi\Phi_N$ defines a 
Hermitian form. Let $U(N)$ be the Unitary group of this 
hermitian form. It is a  quasi-split reductive group over 
$F$ whose group of $F$-points is \[ U(N,F) = \{ g \in 
GL_N(E) \mid g \Phi_N {^t\bar g} = \Phi_N\}. \]  Note 
that $U(N,E) \simeq GL_N(E)$. We have~$U(N)(F_v) \simeq GL_N(F_v)$ 
if $v$ splits in $E$. Otherwise $U(N,F_v)$ 
is a quasi-split unitary group. If $N$ is odd, the classification of hermitian forms by Landherr
\cite{Lan36} implies that 
it is the quasi-split unitary group up to isomorphism. 

\begin{remark}
	We follow Rogawski by denoting by $U(N,F)$ or $U(N)$ 
	the quasi-split group in $N$ variables over a field $F$. 
	The compact unitary group in $N$ variables 
	over $\BR$ will be denoted $U_N(\BR)$. 
\end{remark}
\subsection{The setup}\label{setup}

We now return to $E/F$ being a CM extension as 
before. Let $\fp$ be an ideal of $F$ such that the associated place $v_\fp$ is not in the set $S$ and \numequation \label{condition on p}
\fp \geq 9[F:\BQ]+1.
\end{equation}  We define the subgroups $U(N,\fp^n) \subset U(N, \BAFf)$ to be \[ U(N,\fp^n) := \{ g \in U(N, \hat\MO_F) \subset GL(N, \hat\MO_E) \mid g \equiv I_N \; (\fp^n \MO_E) \}. \] For any finite place $v$ of $F$, let $U(N,\fp^n)_v = U(N,\fp^n) \cap U(N)_v$. At the expense of possibly including additional primes in the set $S$, note that for all $v \notin S \cup \{v_\fp\}$, the subgroup $U(N,\fp^n)_v$ is a hyperspecial maximal compact subgroup of $U(N)_v$. 

Define $G$ to be the
 inner form of $U(3,F)$ defined with respect to a Hermitian inner product and determined by the condition that $G$
 is isomorphic to the compact group $U_3(\BR)$ at the
 infinite places contained in $S_0$, and to the quasisplit
 group $U(3,\BR)$ at the place $v_0$. 
 Since 3 is odd, it follows from Clozel's work in \cite{Clo91} that these conditions do uniquely determine the group
 $G$.  Moreover, for each 
 finite $v$, there are isomorphisms $\phi_v: G_v \to U(N)_v$, 
 canonical up to inner isomorphism. Note that if $F \neq \BQ$ 
 then
 $G$ is by construction anisotropic, an assumption that we make from now on.   We denote by $\fg$
 the complexified Lie algebra of $U(3,\BR)$. 

For each natural number $n$, we fix a compact subgroup $K(\fp^n) = \prod_v K_v(\fp^n)$ of $G(\BAF)$ 
as follows:  at all finite $v \notin S$, we
let $K_v(\fp^n) = \phi_v^{-1}(U(3,\fp^n)_v)$; at $v \in S_f$, 
the subgroup $K_v(\fp^n)$ is an arbitrary open compact 
subgroup independent of $n$; at the archimedean places 
$v \in S_0$, we let $K_v(\fp^n) = G_v$ while $K_{v_0}(\fp^n) \simeq U_2(\BR) \times U_1(\BR)$
is a maximal compact subgroup of $G_{v_0}$. For 
simplicity, we will sometimes use $K_v$ instead of 
$K_v(\fp^n)$ for $v \neq v_\fp$. 

We will denote by $H$ the quasisplit reductive group over $F$ given by the product~$U(2) \times U(1)$. We similarly fix compact subgroups $K^H(\fp^n) =  \prod_v K^H(\fp^n)_v$ with $K^H(\fp^n)_v \simeq U(2,\fp^n)_v \times U(1,\fp^n)_v$ open if $v$ is finite, $K^H(\fp^n)_v$ hyperspecial at $v \notin S \cup \{\fp^n\}$, and $K_v(\fp^n)^H$ a maximal compact subgroup of $H_v$. We again sometimes denote $K_v^H(\fp^n) = K_v^H$ when $v \neq v_\fp$.

\subsection{Test functions}
If $G/F$ is a unitary group, fix Haar measures 
$dg_v$ on the local factors $G_v$, normalized at 
the unramified places so that the groups $K_v(1)$ 
have volume 1. Let $dg =\ten_vdg_v$  be the 
induced measure on $G(\BAF)$.

Let $\omega_v$ be a character  
of the center $Z_v$ of $G_v$. We denote by $C^\infty_c(G_v, \omega_v)$
the space of smooth functions on $G_v$ which are compactly 
supported modulo the center and satisfy \[\tf_v(zg) = \omega^{-1}_v(z)\tf(g), \quad z \in Z_v, g \in G_v.\]
When $v$ is archimedean we additionally require that 
$\tf_v \in C^\infty_c(G_v, \omega_v)$ be $K_v$-finite. We will use the notation $C^\infty_c(G,w)$ for smooth, compactly supported functions on $G(\BAF)$
which are linear combinations of functions $f = \ten_v 
f_v$ with $f_v \in C^\infty_c(G_v, \omega_v)$, and  and 
such that $f_v$ is the characteristic function of $K_v(1)$
for almost all $v$. We will commonly refer to elements of $C^\infty_c(G_v, \omega_v)$ and $C^\infty_c(G,w)$ as local or global test functions.

If $\pi_v$ is an admissible representation of $G_v$ 
on a Hilbert space and $f_v$ is a test function, the 
operator \[ \pi_v(f_v) = \int_{G_v} \pi_v(g)f_v(g) dg_v  \] 
is of trace class, and we will denote its trace by $\tr \pi_v(f_v)$.
We will refer to the functional \[f_v \mapsto \tr \pi_v(f_v)\] as the character
of $\pi_v$. 
Note that if $\pi_v$ admits a central character (e.g.\ 
if $\pi_v$ is irreducible), then $\tr \pi_v$ will
vanish identically on test functions in 
$C^{\infty}(G_v,\omega_v)$ unless this central character is equal to~$\omega_v$.

\subsection{Transfer} \label{partoftheintroabouttransfer}
 The group $H$ appearing in the previous section is an
  endoscopic group of $G$. We work with a fixed endoscopic 
  datum $(H, s, \eta)$ but we will not explicitly make use of the specific $s$ and $\eta$, 
so we refer the reader to \cite[4.2]{R} for details about their construction.
     This endoscopic datum induces a notion of transfer
      of test functions, where a function $ f_v^H \in C^\infty_c(G_v, \omega_v\mu_v^{-1})$
      is said to be a transfer of a function 
      $f_v \in C^\infty_c(G_v, \omega_v)$ if the orbital integrals of
       $f$ and $f^H$ over corresponding conjugacy
        classes
are related by certain identities.

To even make sense of the notion of global transfer, one
has to know that the local transfer of the characteristic 
function of an hyperspecial maximal compact subgroup 
of $G$ is itself the characteristic function of a 
hyperspecial maximal compact subgroup of $H$.  This is 
the content of the  fundamental lemma, proved in the 
case of  $U(2,1)$ by Rogawski, for general unitary 
groups by Laumon-Ng\^o  \cite{LN}, and by Ng\^o \cite{Ngo} in general. We 
cite the version we will need, rephrased from that of 
Rogawski.

\begin{theorem}[\cite{R}, Proposition 4.9.1 (b)]
Let $F_v$ be p-adic, $E_v/F_v$ be unramified and suppose 
the characters $\mu_v$ and $\omega_v$ are unramified. 
Let $f$ be the characteristic function of the subgroup $K_v(1)$. Then the transfer $f^H$ can be chosen to be the characteristic function of $K^H_v(1)$. 
\end{theorem}

Given the above theorem, if we have a function 
$f = \ten_v f_v$ on $G(\BAF)$, we will define its 
transfer to be $f = \ten_v f^H_v$ on $H(\BAF).$ Transfer of functions allows us to compare, and ultimately match up with each other, characters of representations on $G$ and on $H$. These relations are known as character identities, and the bulk of section \ref{identities} will consists in writing them out explicitly in the case of our groups $G$ and $H$. 

\subsection{Automorphic representations}
As discussed in the introduction, we will compute cohomology of arithmetic lattices in terms of automorphic representations, i.e. representations of $G(\BAF)$ appearing in the right-regular representation on~$ \Aut(G(F) \dom G(\BAF),\omega)$ for $\omega$ a fixed central character. 
Because $G$ is anisotropic, this representation constitutes the full automorphic spectrum once $\omega$ is fixed.
It decomposes as a direct sum 
\[ \Aut(G(F) \dom G(\BAF),\omega) = \bigoplus_{\pi} m(\pi)\pi \] over irreducible representations of $G(\BAF)$. Each $\pi$ decompose as a restricted tensor product
 $\pi = \ten'_v \pi_v$, where for each finite place $v$,
the tensor factor $\pi_v$ is an irreducible
admissible smooth representation of $G_v$. 
At almost all finite $v$, the factor $\pi_v$ is furthermore an unramified principal series representation and as such has a vector fixed under the maximal compact $K_v(1)$.  

\subsection{Adelic double quotients and Matsushima's formula} \label{adelic double quotients} Our objects of interest are the arithmetic groups $\Gamma(\fp^n)$, realized as the intersection \[\Gamma(\fp^n) = G(F) \cap \Kfpn \subset G(\BAFf). \] These lattices $\Gamma(\fp^n)$ are the fundamental groups of the compact manifolds
 \[ X(\fp^n) = \Gamma(\fp^n) \dom U(3,\BR)/K_\infty, \] where $K_\infty \simeq U_2(\BR) \times U_1(\BR)$  is 
 the maximal compact subgroup of $U(3,\BR)$. In practice, it will be more convenient to work with the 
 adelic double cosets \[ Y(\fp^n) = G(F)\dom G(\BAF)/\Kpn Z(\BAF)  \] where $Z(\BAF)$ is the center of $G(\BAF)$. 
 These manifolds $Y(\fp^n)$ consist of a finite disjoint 
 union of copies of $X(\fp^n)$. 
\begin{lemma}
	The size of the set of components $\pi_0(Y(\fp^n))$ is bounded by a constant which is independent of the exponent $n$. 
\end{lemma}
\begin{proof}
	We adapt an argument from \cite[\S 2]{MR15}. Considering $G$ as a subgroup of $GL_N/E$, let $\det: G \to U(1,E/F)$ be the determinant map and let $G^1 = \ker (\det)$. This map induces a fibering of $Y(\fp^n)$ over \[ U(1,F) \dom U(1,\BAF) / \det(Z(\BAF)K(\fp^n)).  \] The fibers are adelic double quotients for the group $G^1$, which is simply connected and has a noncompact factor at infinity. So by \cite[7.12]{PR94}, the group $G^1$ satisfies strong approximation with respect to the set $S_\infty$ and $G^1(F) $ is dense in $G^1(\BAFf)$, making the fibers connected. Thus we find that \[ \pi_0(Y(\fp^n)) \simeq U(1,F) \dom U(1,\BAF) / \det(Z(\BAF)K(\fp^n)) = E^1 \dom \BAE^1 / \det(Z(\BAF)K(\fp^n)). \] Now the image $\det (Z(\BAF)) $ is the subgroup $(\BAE^1)^3$ of $\BAE^1$. For each finite place $w$, the factor corresponding to $E_w$ in the quotient $\BAE^1/(\BAE^1)^3$ is a finite set. It follows that by increasing the level in powers of a single prime $\fp$, one can only produce a bounded number of components. 
\end{proof}

We note that automorphic representations appearing in $\Aut(G(F) \dom G(\BAF),1)$, i.e. the representations with trivial central character, can be identified with functions on the quotient $G(F) \dom G(\BAF) /Z(\BA)$. It is those representations which contribute to $H^1(Y(\fp^n),\BC) $. 

By Matsushima's formula \cite{Ma67}, the cohomology of $Y(\fp^n)$ can be computed 
as \[ \dim(H^1(Y(\fp^n), \BC)) = \sum_{\pi=\pi_{\infty}\pi_f}m(\pi)\dim(H^1(\fg,K;\pi_\infty))\dim(\pi_f)^{K_f(\fp^n)}. \] The sum is taken over $ \Aut(G(F)G(\BAF),1)$; it vanishes for almost all $\pi$. We denote by $H^1(\fg,K;\pi_\infty)$ the $(\fg,K)$-cohomology 
of the Harish-Chandra module of $K$-finite smooth vectors in $\pi_\infty$.

\section{Representations with cohomology and endoscopic character identities} \label{identities}
In this section, we first recall the explicit description of 
representations with nontrivial $(\fg,K)$ cohomology
from~\cite{R}.  We then spell out characters identities 
for all local factors which make up the global 
representations with cohomology. The introduction 
of these identities is justified in the latter part of the 
section, where we recall that automorphic
 representations with first cohomology all arise as the 
 transfer of automorphic characters from $H$.

\begin{remark}
A key conceptual point (which however is not made
explicit in the results of~\cite{R}) is that representations 
of $U(2,1)$ with first cohomology correspond
to an Arthur parameter $\psi_{\infty}$ in which the
 $SL_2$ factor maps to the principal $SL_2$ of a 
 Levi subgroup of the form $GL_2 \times GL_1$ 
 in $\hat{G}$.  This means that, if $\psi$
is a global parameter which is has nonzero first cohomology at infinity,
then {\em all} the local factors $\psi_v$ are non-trivial on
the Arthur $\SL_2$, so that the corresponding local factors
$\pi_v$ of members of the packet $\Pi(\psi)$ lie in a non-tempered
local $A$-packet --- assuming the yoga of Arthur parameters and
$A$-packets is correct!  This concrete consequence of
that yoga {\em is} proved in~\cite{R}, and we recall it in
Theorem~\ref{thm:global constraint} below.
\end{remark}

\subsection{Representations with cohomology are the transfers of characters.}
We first recall the results on the representations of $U(3,\BR)$ 
with non-vanishing first $(\fg,K)$ cohomology. We will say that 
an irreducible admissible representation $\pi$ has 
non-vanishing $(\fg,K)$ cohomology if its underlying 
Harish-Chandra module of $K$-finite smooth vectors does. 

\begin{proposition}[\cite{BW}, Proposition 4.11]
	There are two isomorphism classes of representations $\pi$ of $U(3,\BR)$ such that $H^1(\fg,K;\pi) \neq 0$. In both cases we have  $$H^1(\fg,K;\pi) \simeq \BC.$$
\end{proposition}
Before fixing notation for the two representations
 with nonzero first cohomology, we recall a few facts from \cite{R}. To each representation with cohomology (in any degree), Rogawski attaches a triple of integers encoding its Langlands parameter. For representations with $H^1 \neq 0$, it follows from \cite[Section 12.3]{R} that the two triples should respectively be \[ (0,1,-1), \quad \mathrm{and} \quad (1,-1,0). \] These triples are also associated to the parameters of one-dimensional representations of $U(2,\BR) \times U(1,\BR)$, which we will denote $\xi^+ = \xi(0,1,-1)$ and~$\xi^- = \xi(1,-1,0)$. As discussed in \ref{partoftheintroabouttransfer}, the restriction of $\xi^\pm$ to the center of $\lambda$ should be equal to $\mu^{-1}$ (by abuse of notation, we denote by $\mu$  the local character at the infinite place.) By definition of $\mu$, there should be an integer $t$ such that for $z \in \BC^\times$, \[\mu(z) = \left(z/|z|\right)^{2t+1}.\]

Let $\det$ denote the determinant on $U(2)$ and $\lambda$ be the identity embedding of $U(1)$ in $\BC^\times$. With the above normalization, the two characters are \[ \xi^+= \det^{t-1}\ten\lambda, \quad \xi^- = \det^{-t}\ten\lambda^{-1}. \]
Denote $\Xi := \{\xi^+, \xi^-\}.$ We have introduced these characters of $H$ because they satisfy character identities relating them to the representations with nonzero cohomology. Still following \cite[Section 12.3]{R}, we introduce the following notation. Note that the superscript $n$ stands for non-tempered. 
\begin{definition} \label{cohomologicalreps}
	For $\xi \in \Xi$, let $\pi^n(\xi)$ be the representation of $U(3,\BR)$ with non-vanishing first cohomology associated to $\xi$ via the corresponding triple. 
\end{definition}

To each of the two triples is attached a unitary character $\chi^\pm$ of the diagonal Levi subgroup of
$U(2,1)$, and the representation $\pi^n(\xi^\pm)$ appears in the Jordan-Holder
decomposition of the induction $i_G(\chi^\pm)$. 
Finally, to each $\xi$ we associate a discrete series representation which we denote $\pi^s(\xi)$ as in \cite[Section 12.3]{R}. The representation $\pi^s(\xi^\pm)$ is a Jordan-Holder constituent of the representation $i_G(\chi^\mp)$. Although \cite{R} does not 
state it in this language, the representation $\pi^s(\xi)$ is the second member of the Arthur
 packet $\Pi(\xi)$. The elements of $\Pi(\xi)$ satisfy the following character identities.

\begin{proposition}[\cite{R}, Prop. 12.3.3] \label{noncompact char identity}
	Let $\xi \in \Xi$, $\pi^n(\xi)$, and $\pi^s(\xi)$ be as above and let $f$ be a compactly supported smooth function with transfer $f^H$. Then \begin{itemize}
		\item[(a)] $\tr(\xi(f^H)) = \tr(\pi^n(\xi)(f))+\tr(\pi^s(\xi)(f));$
		\item[(b)] $\tr(\pi^n(\xi)(f))-\tr(\pi^s(\xi)(f))$ defines a stable distribution on $G$.
	\end{itemize}
\end{proposition}

This is the archimedean version of the local transfer 
result which will allow us to compare the trace of functions on $G$ with that of their transfers on $H$. 

\subsection{Character identities at the compact archimedean places}  If $\xi \in \Xi$ is as above, and if the group $G$ at the corresponding infinite place is isomorphic to the compact $U_3(\BR)$, let $\Pi(\xi) := \{\bf{1}\}$, the set containing only the trivial representation of $U_3(\BR)$.
These packets satisfy the following character identities. 
\begin{proposition}[Prop. 14.4.2 (c), \cite{R}] \label {compact char identity}
	Let $\xi \in \Xi$ and let $f$ be a test function on $U_3(\BR)$. Let $f^H$ be its transfer to $U(2,\BR) \times U(1,\BR)$. Then \[\tr {\bf 1}(f) = -\tr \xi(f^H) . \]  
\end{proposition}
We now fix some notation which will allow us to treat 
the archimedean places all at once when we discuss 
global phenomena. In short, we will reserve the subscript
 $``\infty"$ for the product of all infinite places. We will denote characters of
 $$U(2,F\ten \BR) \times U(1,F \ten \BR) \simeq [U(2,\BR) \times U(1,\BR)]^d$$ 
 by $\xi_\infty$, and we will let $\Xi_\infty$ denote the set of 
 representations of $[U(2,\BR) \times U(1,\BR)]^d$ such that 
 each local factor is of the form $\xi^\pm_{v}$.

 Given a character $\xi_{v_0}$ of $H_{v_0} = U(2,\BR) \times U(1,\BR)$,
  we will denote by $\pi^n_\infty (\xi_{v_0})$ the representation
   of $G(F \ten \BR) \simeq U(3,\BR) \times U_3(\BR)^{d-1}$
    given by \[ \pi^n_\infty(\xi_{v_0}) := \pi^n(\xi_{v_0}) \ten {\bf1}^{d-1}.\] 

\subsection{Character identities at the non-archimedean places} In this subsection, we construct local packets $\Pi(\xi)$ at the non-archimedean places and recall the character identities they satisfy. 
\subsubsection{Nonsplit places}

Let $E/F$ be a quadratic extension of $p$-adic fields obtained as a localization of our global CM extension and let $G = U(3)$ and $H = U(2) \times U(1)$ be the quasisplit unitary groups of the specified dimensions defined relative to $E/F$. 

Let $\xi$ be a one-dimensional representation of $H$ with central character $\mu^{-1}$. 
Associated to $\xi$ is are two representations with trivial central character: the non-tempered representation $\pi^n(\xi)$, constructed as a quotient of a principal series representation, and the supercuspidal representation $\pi^s(\xi)$.  We will have no use for their explicit description, but we refer the reader to \cite[12.2]{R} and \cite[13.1]{R} for more details about $\pi^n(\xi)$ and $\pi^s(\xi)$ respectively. 
As in the archimedean case, they form an Arthur packet: \[ \Pi(\xi) = \{\pi^n(\xi), \pi^s(\xi)\}   \] which satisfies the following character identity.
\begin{proposition}[Proposition 13.1.4, \cite{R}] \label{padic char identity}
	Let $f \in C^\infty_c(G)$ and let $f^H$ be its transfer. Let $\xi$ be a smooth character of $H$. Then $$\xi(f^H) = \tr \pi^n(\xi)(f) + \tr \pi^s(\xi)(f).$$
\end{proposition} 

\subsubsection{Split places.}
We now describe character identities when $F$ is a 
local field such that $G = GL_3(F)$ and $H = GL_2(F) 
\times GL_1(F)$, being thought of as the Levi subgroup 
of a standard parabolic in $G$. In this case the transfer 
of one-dimensional representations can be described 
explicitly.
\begin{proposition}[Proposition 4.13.1 (b), \cite{R}] \label{split char identity}
	Let $\xi$ be a one-dimensional representation of $H$, and let $i_G(\xi)$ be the unitary parabolic induction of $\xi \ten \mu \circ \det$. Then the following holds for any test function $f$ on $G$: \[ \tr(i_G(\xi))(f) = \tr(\xi(f^H)). \] 
\end{proposition}The proposition is not entirely proved in \cite{R}, and the missing elements can be found in \cite{va72}. Since it is induced from the group $GL_2 \times GL_1$, the representation $i_G(\xi)$ is irreducible \cite{BZ77}.  In this situation, we define $\Pi(\xi) := \{i_G(\xi)\}.$
\subsection{Global Constraints}
Here, we justify the introduction of the local
character identities by recalling the result of \cite{R} which states 
that representations which are of the form $\pi^n(\xi_v)$ 
at one place must be globally the transfer of a 
one-dimensional representation. 

Now working globally, let  $\xi$ be a one-dimensional 
representation appearing in $L^2_{\rm disc}(H(F)\dom 
H(\BAF))$ and write $\xi = \ten_v' \xi_v$. In the above, 
we have described an Arthur packet $\Pi(\xi_v)$ for 
each $\xi_v$. Let \[\Pi(\xi) = \{ \ten'_v \pi_v \mid \pi \in \Pi(\xi_v), \; \pi = \pi^n(\xi_v) \text{ for almost all }v.\}\]
The following result provides constraints on which global representations can have non-vanishing first cohomology at the infinite place. 
\begin{theorem}[\cite{R}, 13.3.6 (c)]
	\label{thm:global constraint}
	If $\pi$ is a discrete automorphic representation of 
	$G$ such that $\pi_v$ is of the form $\pi^n(\xi_v)$ 
	for some place $v$ of $F$ which does not split in 
	$E$, then $\pi \in \Pi(\xi)$ for some one-dimensional 
	$\xi \in \Pi(H)$.
\end{theorem}
\begin{corollary} \label{global coho reps come from characters}
If $\pi = \ten'_v \pi_v$ is a discrete automorphic representation of $G$ such that at the non-compact infinite place $v_0$, the representation $\pi_{v_0}$ satisfies $H^1(\fg,K; \pi_{v_0}) \neq 0$, then $\pi \in \Pi(\xi)$ for some global character $\xi$ with $\xi_{\infty} \in \Xi_\infty$. 
\end{corollary}
We now have established all the necessary local character identities, as well as the fact that representations with first cohomology occur in global Arthur packets~$\Pi(\xi)$. 
\section{Choice of test functions} \label{tf} 
We now explain which test functions we will use to compute growth of cohomology of lattices in $U(3,\BR)$ using the trace formula. If $K_f$ is any compact open subgroup of $G(\BAFf)$, let $\Gamma = K_f \cap G(F)$, and let $X_\Gamma = G(F) \dom G(\BAF) / K_f K_\infty Z(\BA)$.  Recall that $v_0$ is the infinite place at which $G$ is noncompact.  By ``Matsushima's formula'',
as recalled above, and the discussion leading to Definition \ref{cohomologicalreps}, we have 
\nummultline
\label{Matsushima}
\dim(H^1(X_\Gamma, \BC)) = \sum_{\pi}m(\pi)\dim(H^1(\fg,K;\pi_{v_0}))\dim(\pi_f)^{K_f} \\ = \sum_{\xi_{v_0} \in \Xi_{v_0}} \sum_{\pi_{\infty} = \pi^n_\infty(\xi_{v_0})} m(\pi)\dim(\pi_f)^{K_f} 
\end{multline}

We wish to compute, or at least estimate, the right hand side 
of this formula using the trace formula applied to an appropriate test function depending on the subgroup $K_f$. 

We will specialize $K_f$ to be $\Kfpn = K(\fp^n) \cap G(\BAFf)$. We explain how the last sum of \eqref{Matsushima} can be realized as the trace of a global test function $\tf$ on a subspace of $\Aut(G(F)\dom G(\BAF))$ and we discuss how this relates to Arthur's stable trace formula. The goal of this subsection is the proof of the following statement is to construct a global test function $\tf(n) =\tf_\infty \tf_f(n)$ whose trace computes the dimension of cohomology, i.e. that satisfies \[ \sum_{\xi_{v_0} \in \Xi_{v_0}} \sum_{ \pi_\infty = \pi_\infty^n(\xi_{v_0})} m(\pi)\dim(\pi_f)^{\Kfpn} = \sum_{\xi : \xi_\infty \in\Xi_\infty} \sum_{\pi \in \Pi(\xi)} m(\pi) \tr \pi(\tf(n)). \]

The remainder of this subsection is devoted to
constructing the various factors of the desired test function $f(n)$.

\subsubsection{Archimedean places} We first discuss the function $f_\infty$ on $G(F_\infty)$. As above will let $F_\infty$ be the product of all the infinite completions. Recall that we are denoting by $v_0$ the unique infinite place at which $G$ is non-compact, and by $S_0$ the set consisting of all infinite $v \neq v_0$.  Then the group $G$ so that $G_{v_0} \simeq U(3, \BR)$ and $G_v = U_3(\BR)$ is compact for $ v \in S_0$. 

We will construct a function $\tf_\infty = \prod_{v \mid \infty} f_v$. For $v \in S$, we let $f_v$ be the constant function \[ f_v := \frac{1}{\Vol( U_3(\BR))}. \]  In contrast, the choice of $f_{v_0}$ is not constructive.  Recall that we will be taking a trace over packets of the form $\Pi(\xi^\pm_{v_0})$, and that we only wish to count the contribution of the non-tempered representations $\pi^n(\xi^\pm_{v_0})$. Thus the smooth,  compactly supported function $f_{v_0}$ must satisfy \[ \tr \pi^n(\xi^\pm_{v_0})(f_{v_0}) = 1, \quad \tr\pi^s(\xi^\pm_{v_0})(f_{v_0}) = 0. \] Such an $f_{v_0}$ is guaranteed to exist by the linear independence of characters, as proved by Harish-Chandra and which we recall here. 
\begin{theorem}[Theorem 6, \cite{HC}]
	Let $\pi_1,...,\pi_n$ be a finite set of quasi-simple irreducible representations of G on the Hilbert spaces $H_{\pi_1},...,H_{\pi_n}$ respectively. Suppose none of them are infinitesimally equivalent. Then their characters $ \tr {\pi_1},..., \tr {\pi_n}$ are linearly independent.
\end{theorem}
The above characters are defined as distributions on the space of compactly supported smooth functions. This allows us to infer the following. 
\begin{corollary} \label{see infinite}
	There is a smooth function $f_{v_0}$ on $U(3,\BR)$ with the property that for  $\xi^\pm_{v_0}$ we have \[ \tr \pi^n(\xi^\pm_{v_0})(f_{v_0}) = 1, \quad \tr \pi^s(\xi^\pm_{v_0})(f_{v_0}) = 0. \]
\end{corollary}
Accordingly, we define the function $f_\infty$ as \[ \tf_\infty = \prod_{v \mid \infty} \tf_{v} \in C^\infty_c(G(F_\infty),1).  \]  By construction and the structure of the packets $\Pi(\xi_v)$ for $v \mid \infty$ described in section \ref{identities}, we have:
\begin{lemma} \label{see infinite but all of them}
	Let $\pi_\infty = \prod_{v \in S_\infty} \pi_v$  with $\pi_v \in \Pi(\xi_v^\pm)$.  Then \[ \tr \pi_{\infty}(f_\infty) = \begin{cases}
	1 & \text{ if }\pi_\infty = \pi^n_\infty(\xi_{v_0}) \\ 
	0 & \text { otherwise. }
	\end{cases} \]
\end{lemma}
So far we have produced test functions detecting the representations with nontrivial first cohomology at the noncompact place among those belonging to the packets $\Pi(\xi_\infty)$. 

\subsubsection{Non-archimedean places.}\label{test function at the non-archimedean places}  
We want to produce a function which counts the dimension of the space of vectors fixed by the subgroup \[\Kpn_f = K_{v_\fp}(\fp^n) \times \prod _{v \neq v_\fp \text{ finite}} K_v, \] introduced in \ref{setup}. 
We first consider the non-archimedean places $v \neq v_\fp$.  For each of these, we let $$\tf_v := \frac{1_{K_v}}{\text{Vol}(K_v)},$$ be the indicator function of $K_v$, scaled so that $$ \tr \pi_v (f_v) = \dim \pi_v^{K_v}. $$ Note that this dimension will be equal to $1$ in the cases where $K_v=K_v(1)$ and $\pi_v$ is an unramified principal series. 

As for the place $v=v_\fp$, at level $\fp^n$ we use the scaled indicator function of the compact subgroup of the corresponding level: $$f_{v_\fp}(n) := \frac{1_{\Kpn}}{\text{Vol}(\Kpn)}. $$
We recall the following property of these various $f_v$.

\begin{lemma} \label{see finite} 
	Let $\pi_f = \prod_v \pi_v$ be a representation of $G(\BAFf)$ such that $\pi_v$ is an unramified principal series representation for almost all $v$, and let $$\tf_f(n)  = \tf_{v_\fp}(n) \cdot \prod_{v \neq v_\fp} \tf_v.$$ Then  \[ \tr \pi_f(\tf_f(n)) = \dim\left(\pi_f\right)^{\Kpn}.   \]
\end{lemma}

We collect our local constructions at all places and define the global test function $\tf(n)$ as the product $\tf(n) = \tf_\infty\cdot \tf_f(n)$. 

 \begin{proposition}\label{transtf}
	The global test function $\tf(n)$ defined as above satisfies \[ \sum_{\xi_{v_0} \in \Xi_{v_0}} \sum_{ \pi_\infty = \pi_\infty^n(\xi_{v_0})} m(\pi)\dim(\pi_f)^{\Kfpn} = \sum_{\xi : \xi_\infty \in\Xi_\infty} \sum_{\pi \in \Pi(\xi)} m(\pi) \tr \pi(\tf(n)). \]
\end{proposition}
\begin{proof}
First recall that from the characterization of cohomological representation in Definition \ref{cohomologicalreps}, the representations $\pi$ that contribute to degree 1 cohomology of lattices in $G_{v_0} \simeq G_\infty / (\prod_{v \in S_0} G_v)$, must be of the form $\pi_\infty= \pi^n_{\infty}(\xi_{v_0})$ with $\xi_{v_0} \in \Xi_{v_0}$. Next, by Corollary \ref{global coho reps come from characters} and Proposition \ref{compact char identity}, all such representations live in packets of the form $\Pi(\xi)$ such that $\xi_{\infty} \in \Xi_\infty$. Lemma \ref{see infinite but all of them} shows that among the representations belonging to the packets $\Pi(\xi)$ with $\xi_{\infty} \in \Xi_\infty$, the trace of $f_\infty$ detects only the ones contributing degree 1 cohomology. Finally by Lemma \ref{see finite}, the test function $f_f(n)$ computes exactly the dimension of the space $\Kfpn$-fixed vectors. 
\end{proof}

\section{Rogawski's stable trace formula}  Here, 
we recall results of Rogawski on the stabilization of the 
distribution \[ \sum_{\pi \in \Pi(\xi)} m(\pi) \tr(\pi(f)),  \] for 
an automorphic character $\xi$ of $H$. These hold for 
an arbitrary $f \in C^\infty_c(G(\BAF),\omega)$. We will 
later choose $f$ to be the function $f(n)$ constructed in Section \ref{tf}. 

\subsection{Epsilon factor and pairings} Let $\xi = \ten \xi_v$ be an automorphic character of $H$. Following Rogawski, we define the set $\hat{\Pi}(\xi):=\{1,\xi\}$, where $1$ and $\xi$ need only to be thought as formal objects. The identities satisfied by the local Arthur packets from section \ref{identities} can be packaged in the following local pairings with $\hat \Pi(\xi)$. 
\begin{definition}Let $\xi = \ten \xi_v$ and let $\Pi(\xi_v)$ and $\Pi(\xi)$ be the corresponding local and global Arthur packets. We define the pairing $\ip{\cdot}{\cdot}: \hat{\Pi}(\xi) \times \Pi(\xi_v) \to \{\pm 1\}$. \begin{itemize}
		\item[(a)] If $v$ is ramified or inert in $E$ and $G_v$ is quasisplit, the pairing is given by \begin{align*} \ip{1}{\pi^n(\xi_v)} = 1, \quad \ip{1}{\pi^s(\xi_v)} = -1, \quad \ip{\xi}{\cdot} \equiv 1. 
		\end{align*}
		\item[(b)] If $v$ is split in $E$, let $\ip{\cdot}{\cdot} \equiv 1.$
		\item[(c)] If $G_v = U_3(\BR)$, let $\ip{\cdot}{\cdot} \equiv -1.$
	\end{itemize}
\end{definition}  This extends to automorphic representations $\pi = \ten' \pi_v$ through \[\ip{\epsilon}{\pi} = \prod_v \ip{\epsilon}{\pi_v}, \quad \epsilon \in \{1,\xi\}.\] Note that $\ip{1}{\pi} = (-1)^{n(\pi)+d-1}$, where $n(\pi)$ is the number of places where $\pi_v = \pi^s(\xi_v)$ and $d$ is the degree of $F/\BQ$.

\begin{definition}\label{trpi}Let $\xi = \ten \xi_v$ be an automorphic character of $H$, and let $\Pi(\xi_v)$ an $\Pi(\xi)$ be the corresponding local and global Arthur packets. Let $f = \prod_v f_v$ be a factorizable test function. 
	\begin{itemize}
		\item[(a)] For local packets $\Pi(\xi_v)$, let  \[\tr \Pi_v(\tf_v) = \sum_{\pi_v \in \Pi(\xi_v)} \ip{1}{\pi_v}\tr \pi(\tf_v).\] Note that for $v$ split in $E$ or for $G_v$ compact, the packet $\Pi(\xi_v)$ consists of a single representation. 
		\item[(b)] For the global packet $\Pi(\xi)$ define \[ \tr\Pi(\tf) = \prod_v \tr \Pi_v(\tf_v) = \sum_{\pi \in \Pi(\xi)} \ip{1}{\pi}\tr\pi(\tf). \]
	\end{itemize} 

\end{definition}

In order to state the stabilization theorem of Rogawski, we now introduce
the global root number, $\epsilon(\frac12, \phi)$
 associated to $\xi$. We first recall the definition of
  $\phi$ from \cite{LR}. The automorphic character
  $\xi$ decomposes as a product $\xi = \chi_1 \ten (\chi_2 \circ \det)$ on $U(1,\BA) \times U(2,\BA)$.
   We promote $\chi_1$ to a character $\chi_{1,E}$ of
    $\BAE^\times/E^\times$ given by $\chi_{1,E}(\alpha) = \chi_1(\alpha/\bar\alpha)$. 
The character $\phi$ is then defined as $\phi(\alpha) := \mu(\alpha) \chi_{1,E}(\alpha)$ and the global root number is the value at $s = \frac{1}{2}$ of the epsilon factor $\epsilon(s,\phi)$; it takes values in $\{\pm 1\}$.
\subsection{Stability over $A$-packets} We can now state the following key result of Rogawski on the multiplicities $m(\pi)$ of representations $\pi \in \Pi(\xi)$. It does not appear exactly in this form in \cite{R}, but we explain how to obtain it in the remark below.
\begin{theorem} \label{tfxi}  Let $\Pi(\xi)$ be a packet associated to a $1$-dimensional representation~$\xi$ of $U(2) \times U(1)$, and let $f = \prod f_v$ be a factorizable test function. Then we have \numequation \label{stable trace formula} \sum_{\pi \in \Pi(\xi)} m(\pi)\tr(\pi)(f)= \frac{1}{2} \epsilon(\frac12,\phi)\tr(\Pi(f)) + \frac{1}{2}  \tr(\xi(f^H)). \end{equation}
\end{theorem}
\begin{remark}This result appears as \cite[Equation (14.6.3)]{R} in the proof of Theorem 14.6.4, \emph{but without the factor of $\epsilon(\frac{1}{2}, \phi)$}. It is a specialization of Rogawski's stabilization of the trace formula \cite[Theorem 14.6.1]{R} to the case of non-tempered packets. After being informed that the $\epsilon$-factor was not identically equal to $1$, Rogawski wrote the erratum \cite{LR}, in which he proves the correct multiplicity formula for non-tempered Arthur packets. Right above \cite[Theorem 1.2]{LR}, he explains how \cite[Equation (14.6.3)]{R} should be modified, transforming it into \eqref{stable trace formula}. \end{remark}

As we said above, the theorem Rogawski is working towards when he writes down \eqref{stable trace formula} is \cite[Theorem 14.6.4]{R}, and we next state the corrected version from \cite{LR}.  Here, $n(\pi)$ is the number of places at which $\pi = \pi^s(\xi)$ and $N$ is the number of infinite places $v$ such that $G_v$ is compact.  Note that we chose $G(\BR)$ so that $N = d-1$. 
\begin{theorem}[Theorem 1.2, \cite{LR}] \label{multiplicity theorem}
	Let $\pi \in \Pi(\xi)$. Then $m(\pi) = 1$ if $(-1)^{n(\pi)+N} = \epsilon(1/2,\phi)$ and $m(\pi)=0$ otherwise. In other words \numequation \label{smf} m(\pi) = \frac{1}2\left( \epsilon(\frac12,\phi)(-1)^{n(\pi)+N}+1\right).\end{equation}
\end{theorem}

We now describe the main idea of the proof of Theorem \ref{multiplicity theorem}, which is a culmination of \cite{R}, in the case of non-tempered packets. Rogawski deduces it by expanding each of the two summands of \eqref{stable trace formula} into their local factors, using the character identities we have recalled in Propositions~\ref{noncompact char identity}, \ref{compact char identity}, \ref{padic char identity} and ~\ref{split char identity}. The result is: 

\nummultline \label{open the bowels and see the character identities}
\sum_{\pi \in \Pi(\xi)}  m(\pi) \tr(\pi(f)) =\\\quad\quad\quad   \frac{(-1)^N}{2}  \epsilon(\frac{1}{2},\phi)\prod_{v \in S_0} \tr{\bf 1}(f_v)\prod_{v \text{ split}} \tr i_G(\xi)(f_v)\prod_v \{ \tr(\pi^n(\xi_v)(f_v))-\tr(\pi^s(\xi_v)(f_v))\} \\\quad\quad\quad\quad +\frac{(-1)^N}{2} c \prod_{v \in S_0} \tr{\bf 1}(f_v)\prod_{v \text{ split}} \tr i_G(\xi)(f_v) \prod_v\{ \tr(\pi^n(\xi_v)(f_v))+\tr(\pi^s(\xi_v)(f_v))\}. 
\end{multline} 
In each term, the first product is over all archimedean $v$ for which $G_v$ is compact, the second is over all $v$ which split in $E$,  and the third is over all remaining ramified and inert $v$.  The constant $c$ is a global transfer factor and in the course of the proof, Rogawski shows that $c = (-1)^N$.

\begin{remark}\label{crucial observation}
	Importantly, we can make a slightly more detailed statement than Theorem \ref{multiplicity theorem}: for a given automorphic representation $\pi \in \Pi(\xi)$, the term which always  contributes a `` $1$" in the multiplicity formula of theorem \ref{multiplicity theorem} is the second term of the sum~\eqref{open the bowels and see the character identities}, i.e. the one corresponding to $H$. Indeed, as mentioned above, the global transfer $c$ is always equal to $(-1)^N$. Thus at the factors where the two summands differ, the $H$-summand always contributes a $+1$ whereas the sign of the first summand may vary. More precisely, 
	the first summand contributes a $\pm 1$ to the multiplicity formula for $\pi$ depending on whether or not $\epsilon(\frac12,\phi) = (-1)^{n(\pi)+N}.$ 
\end{remark}

\begin{remark}
	In \eqref{open the bowels and see the character identities}, we made an additional modification from what is written down in the proof of Theorem 14.6.4 of \cite{R}. When he splits his expression into local factors, Rogawski does not take the split primes into account, whereas they should be included in the product. We have thus modified added the factors at the split primes, which satisfy the character identities from Lemma \ref{split char identity}. 
\end{remark}

\subsection{The trace is bounded by the contribution of $H$.}
Here  we use the stability results of Rogawski recalled in the previous section to bound the trace of our test $f(n)$ over the packets $\Pi(\xi)$ by the contribution of the endoscopic group $H$. We start with an test arbitrary function $f$ with nonnegative trace.

\begin{theorem} \label{ineq}
	 Let $\tf = \ten_v f_v$ be a test function such that $\tr(\pi)(f) \geq 0$ for all \[ \pi \in \{\pi \in \Pi(\xi) \mid \xi_\infty \in \Xi_\infty\}.\] Then  $$ \sum_{\xi : \xi_\infty \in\Xi_\infty} \sum_{\pi \in \Pi(\xi)}m(\pi) \tr(\pi)(\tf) \leq 2 \sum_{\xi : \xi_\infty \in\Xi_\infty} \tr(\xi)(\tf^H). $$ 
\end{theorem}
\begin{proof}
Fix a character $\xi$. Introduce the temporary notation to package the contributions of the split and compact places: $$\const(\xi) := \prod_{v \in S_0} tr{\bf 1}(\tf_v)\prod_{v \text{ splits in }E} \tr i_G(\xi)(\tf_v). $$
By Theorem \ref{tfxi} we have: $$\sum_{\pi \in \Pi(\xi)} m(\pi)\tr(\pi)(f)= \frac{1}{2} \epsilon(\frac12,\phi)\tr(\Pi(\xi)(f)) + \frac12 \tr(\xi(f^H)).$$  Looking at the two right-hand side summands in more detail as in \eqref{open the bowels and see the character identities}, we see that  \begin{align*}\sum_{\pi \in \Pi(\xi)} m(\pi)\tr(\pi)(f) =  \frac{\const(\xi)}{2} ( &\epsilon(\frac{1}{2},\phi)(-1)^N\prod_v \{ tr(\pi^n(\xi_v)(f_v))-tr(\pi^s(\xi_v)(f_v))\} \\ &+ \prod_v\{ tr(\pi^n(\xi_v)(f_v))+tr(\pi^s(\xi_v)(f_v))\}). \end{align*} 
Here the products are taken over all places of $F$ not contributing to $c(\xi)$. Here we repeat the observation of Remark \ref{crucial observation}: the terms in the first summand of the right-hand side are identical to those in the second summand, with the exception of a possible factor of $-1$. But our assumption of $f$ is that all the local traces are non-negative. Thus for each $\xi$ the second summand dominates and we have $$ \sum_{\pi \in \Pi(\xi)} m(\pi)\tr(\pi)(f)  \leq 2 \tr(\xi(f^H)). $$ We conclude by summing over all $\xi$ with $\xi_\infty \in \Xi_\infty$. 
\end{proof}
Of course the function $f(n)$ constructed in the the previous section has positive trace, which gives us: 
\begin{corollary}
Let $f(n)$ be the function from section \ref{tf}. Then  $$ \sum_{\xi : \xi_\infty \in\Xi_\infty} \sum_{\pi \in \Pi(\xi)}m(\pi) \tr(\pi)(\tf(n)) \leq 2 \sum_{\xi : \xi_\infty \in\Xi_\infty} \tr(\xi)(\tf(n)^H). $$
\end{corollary}
\begin{remark}
	To get a lower bound, we would need to show that in a positive proportion of cases, $m(\pi)=1$. This would follow from even a coarse understanding of the distribution of $\epsilon(\frac 12, \phi)$, which we hope to achieve in the near future.  
\end{remark}
\section{Bounding the growth of cohomology}
Here we use the inequality of Theorem \ref{ineq} to bound the growth of cohomology for lattices of the form $\Gamma(\fp^n)$.  This first involves describing the transfer $f(n)^H$, which then allows us to interpret and bound the sum $$\sum_{\xi} \tr(\xi)(\tf(n)^H).$$ 

\subsection{Transfer}
We first describe local transfers. Recall that at all places $v \neq v_\fp$ of $F$ we have made a choice of a compact subgroup $\Kv \subset G_v$ and a corresponding choice of subgroup $\Kv^H \subset \Gv^H$. In the case where $v \notin S$, the group $K_v$ (resp. $K_v^H$) is the unramified compact open subgroup $K_v(1)$ (resp. $K_v(1)^H$). The following statement is a consequence of the fundamental lemma. 
\begin{lemma}
	Let $v \notin S$ and let $\tf_v = {\bf1}_{\Kv}$. Then the function $f^H$ appearing in the statement of Proposition \ref{padic char identity} is $\tf_v^H = {\bf1}_{\Kv^H}$. 
\end{lemma}
For $v = v_\fp$, we want to know not only the transfers of the characteristic functions of maximal compacts, but also these same transfers for the characteristic functions
of congruence compact open subgroups $K_{v_\fp}(\fp^n)$. In this case, Ferrari \cite{Fe} has concretely described the function $f(n)^H_{v_p}$. His result assumes that the residue characteristic of $\fp$ is strictly greater than $9d+1$, i.e. \emph{assez grande}. 
\begin{lemma}[Theorem 3.2.3, \cite{Fe}]\label{Ferrari}
Let $\norm(\fp)$ be the residue characteristic of~$\fp$. The functions $$ \tf(n)_{v_{\fp}} = \frac{1_{K_{v_\fp}(\fp^n)}}{{\rm vol}(K_{v_\fp}(\fp^n))} $$ and
 $$  f(n)^H_{v_\fp} = \norm(\fp)^{-2n}\frac{1_{K_{v_\fp}(\fp^n)^H}}{{\rm vol} (K_{v_\fp}(\fp^n))} $$ are a transfer pair. 
\end{lemma}
\begin{remark} \label{rescale} We make a few observations: \begin{itemize}
\item[(i)] The $2$ appearing in the exponent is the quantity $$d(G,H) = \frac{(\dim G - \dim H)}{2}$$ four our specific choice of $H$ and $G$. Ferrari's result applies for a more general choice of a group $G$ and an endoscopic group $H$.  
\item[(ii)] The factor $\epsilon_{G,H}$ appearing in \cite[Theorem 3.2.3]{Fe} is equal to $1$ in this particular case. It suffices to choose tori $T$ and $T_H$ in $G$ and $H$, for example by choosing diagonal tori as Rogawski does, so that the restriction of the two Galois actions agree. 
\item[(iii)] The transfer function $$\tf(n)^H_{v_\fp} = \norm(\fp)^{-2n}\frac{1_{K_{v_\fp}(\fp^n)^H}}{{\rm vol} (K_{v_\fp}(\fp^n))}$$ is not scaled so that its trace counts the dimension of $\Kvppn^H$-fixed vectors in a representation: this would be accomplished by the function \[ \frac{\tf(n)^H_{v_\fp} \norm(\fp)^{2n} {\rm vol }(K_{v_\fp}(\fp^n))}{{\rm vol }(K_{v_\fp}(\fp^n)^H)}. \] It is known that \begin{align*}
{\rm vol }(K_{v_\fp}(\fp^n)) &\asymp \norm(\fp)^{-9n} \\
{\rm vol }(K^H_{v_\fp}(\fp^n)) &\asymp \norm(\fp)^{-5n}, 
\end{align*} 
following the formulas for cardinality of unitary groups over finite fields.
As such, if $\tfn_{v_\fp}$ is the function counting the dimension of fixed vectors under $\Kvppn^H$ in a representation, we have \[ \tfn_{v_\fp} \asymp \frac{\tf(n)^H}{\norm(\fp)^{2n}}. \]
\end{itemize}

\end{remark}

We now discuss bounding the trace of the transfer for $v \in S_f$.

\begin{lemma} \label{bernstein}
	Let $\tf = f(n)$ be as in section \ref{tf}. The product $\prod_{v \in S_f} \xi_v(\tf_v^H)$ can be bounded above, uniformly over the  of $\xi$ such that $\xi_\infty \in \Xi_\infty$. 
\end{lemma}
\begin{proof}
The set $S_f$ is finite so it suffices to give a bound independent of $\xi$ for each $v \in S$. 
For each such $v$ we have \[\xi_v(f_v^H) = \tr \pi^n(\xi_v)(f_v) + \tr \pi^s(\xi_v)(f_v).\] The trace of the function $f_v$ counts the number of $K_v$-fixed vectors. By Bernstein's uniform admissibility theorem, see for example \cite{C}, there is an integer $c_{K_v}$ depending only on $K_v$, such that $0 \leq \dim \pi_v^{K_v} \leq c_{K_v}$ for all smooth admissible representations $\pi_v$ of~$G_v$.  Thus $\xi_v(f_v^H) \leq 2 c_{K_v}$. 
\end{proof}
Finally we consider the infinite places. Recall the choice of test functions made in \ref{test function at the non-archimedean places}, together with the character identities of section \ref{identities} and the discussion below \eqref{open the bowels and see the character identities}. These allow us to conclude that for each character $\xi_\infty \in \Xi_\infty$, we have \begin{equation}\label{the product of the infinite traces is one}\prod_{v \mid \infty} \tr(\xi_v(f_v^H)) =  \tr \pi^n(\xi_{v_0})(f_{v_0}) \times \prod_{v \in S_0} \tr \textbf{1}(f_v) = 1.\end{equation}

\subsection{Growth} We now give the promised bounds on the growth of degree $1$ cohomology of lattices in $U(2,1)$. From this point on, our argument proceeds very much along the same steps as Marshall's \cite{Ma14}. 
Recall that
the automorphic representations $\xi$ which transfer to 
our $\pi$ are characters
of $U(2,\BA) \times U(1,\BA)$.
Having fixed the id\`ele class character $\mu$ of $F$ and 
a corresponding integer $t$ as in \cite[$\S$ 12.3]{R}, 
we find that their archimedean components 
lie in the set
\[\Xi := \{\det^{t-1}\ten\lambda, \det^{-t}\ten\lambda^{-1}\}; \] here $\det$ 
denotes the determinant on $U(2)$ and $\lambda$ 
is the identity embedding of $U(1)$ in~$\BC^\times$. 
In the following lemma, we compute the asymptotics 
of the dimensions of $K^H(\fp^n)$-fixed vectors using 
normalized indicator functions as we did for $G$. 
\begin{remark}
Note that the function $\tfn$ appearing in the following 
lemma is asymptotically a re-scaling of $\tf(n)^H$ by the 
factor discussed in \ref{rescale}: \[ \tfn \asymp \frac{\tf(n)^H}{\norm(\fp)^{2n}}. \] 
We introduce this new function in order to separate the two 
sources of growth: one coming from the re-scaling of $f(n)$ and
which we ignore for the moment, and one coming from 
multiplicities of representations on $H$, which 
is the subject of the following lemma. 
\end{remark}

\begin{lemma} \label{growthofthesmallthing}
Denote by $f_v$ the local components at $v \neq v_\fp$ of the function $f(n)$. Let $\tfn =  \prod_v \tfn_v$ be defined as
 \[\tfn_v = \begin{cases} \tf^H_v & v \neq v_\fp \\ \frac{\mathbf{1}_{\Kvppn^H}}{\Vol(\Kvppn)^H} & v = v_\fp. \end{cases}\]

If we write
 $$\Trn := \sum_{\xi \text{
s.t. } \xi_\infty \in \Xi_\infty}\xi(\tfn) ,$$ then $\Trn
\ll \norm(\mathfrak p)^n .$  
\end{lemma}
\begin{proof}
	By construction, we have 
	$$ \xi(\tfn) =  \prod_{v | \infty} \xi_v(\tfn_v) \cdot \prod_{v<\infty} \xi_v(\tfn_v).$$ 
As discussed in \eqref{the product of the infinite traces is one} and by the definition of $\tfn_v$ at the infinite places, we have that $ \prod_{v | \infty} \xi_v(\tfn_v) =1$ for any character $\xi$ such that $\xi_{\infty} \in \Xi_\infty$, so $$  \xi(\tfn) = \prod_{v<\infty} \xi_v(\tfn_v). $$ 
Recall that $\xi^f :=\prod_{v<\infty} \xi_v$ is a character on $H(\BAF^f) = U(1,\BA^f_F) \times U(2,\BA^f_F) $ 
with prescribed central character $\mu$. 
Any such $\xi^f$ is a product of the form 
$\chi_1 \ten \chi_2 \circ \det $, where $\chi_1$ 
and $\chi_2$ are characters of $U(1,\BA^f_F)$
such that $\chi_1 \chi_2^2 = \mu$.
Now $U(1,\BA^f_F)$ equals $(\BA_E^f)^{\mathrm{Nm} = 1},$ 
	the subgroup of norm one elements in the group of finite id\`eles of $E$. 
Thus our sum over $\chi$ is a sum over the 
characters $\chi_2$ of $U(1,\BA_E)$, or (which is irrelevant here,
but better from the perspective of generalizations) proportional to a sum
over the characters
$\chi_1$ of $U(1,\BA_E)$.
 The trace of $\prod_{v \not\in S} \tfn_v$
on these characters vanishes unless the prime-to-$S_f$ part of
the conductor of $\chi_1$ divides $\fp^n$,
in which case it has trace $1$.

We saw in \ref{bernstein}
that the product $\prod_{v \in S_f} \xi_v(\tf_v^H)$ is 
bounded above uniformly in~$\xi$. Thus the sum $\Trn$ is bounded above by a multiple
of the number of characters of~$U(1,\BA_E)$ of
conductor dividing $\fp^n$, which is asymptotically 
proportional to~$\norm(\mathfrak p)^n$.

\end{proof}
\begin{remark}
A stronger result than this should hold: it is likely 
the case that either $\Trn \asymp \norm(\fp^n)$, 
or $\Trn = 0$ for all $n$, i.e. the family $\Gamma(\fp^n)$ 
has vanishing first cohomology for all $n$. 
\end{remark}

\begin{theorem}
Let $\Gamma(\fp^n) = \Kpn\cap G(F)$. Then \[ \dim H^1(\Gamma(\fp^n), \BC) \ll N(\fp)^{3n}. \]
\end{theorem}
\begin{proof}
Following the discussion in \ref{adelic double quotients}, 
we see that  \[ \dim H^1(\Gamma(\fp^n), \BC) \asymp H^1(Y(\fp^n), \BC),\] where the $Y(\fp^n)$ are disconnected adelic quotients. Indeed, the $Y(\fp^n)$ consists of a finite number of copies of $X(\fp^n) = \Gamma(\fp^n) \dom U(3,\BR)/K_\infty$, whose number is bounded independently of $n$. By Matushima's formula, 
this latter quantity can be computed as
\[ \dim H^1(Y(\fp^n), \BC) = \sum_{\pi_\infty = \pi^n_\infty(\xi)} \prod_{v \in S_f} \xi_v(\tf_v^H)(\pi)\dim(\pi_f)^{\Kfpn}. \]
Proposition~\ref{transtf} constructs a test function $\tf(n)$ satisfying
\[ \sum_{\pi_\infty = \pi^n_\infty(\xi)} m(\pi)\dim(\pi_f)^{\Kfpn} =
\sum_{\xi : \xi_\infty \in\Xi_\infty} \sum_{\pi \in \Pi(\xi)} m(\pi) \tr \pi(\tf(n)),\]
and Theorem~\ref{ineq} gives the inequality
\[\sum_{\xi : \xi_\infty \in\Xi_\infty} \sum_{\pi \in \Pi(\xi)}m(\pi) \tr \pi(\tf(n))
\leq \frac{1}{2} \sum_{\xi} \tr(\xi)(f^H(n)). \]
As discussed in \ref{rescale}, the sum on the right-hand side
differs from the expression $\Trn$ of Lemma \ref{growthofthesmallthing} 
by a factor of $\norm(\mathfrak p)^{2n}$.
Putting the preceding displayed formulas together with 
Lemma~\ref{growthofthesmallthing},
we find that \[ \sum_{\pi_\infty = \pi^n_\infty(\xi)} m(\pi)\dim(\pi_f)^{K_f} \ll \norm(\mathfrak p)^{2n}\Trn \ll \norm(\mathfrak p)^{3n},   \]
as claimed.
\end{proof}

\bibliography{endoscopy}{}
\bibliographystyle{plain}
\end{document}